\documentclass[10pt,a4paper]{amsart}
\usepackage[utf8]{inputenc}
\usepackage[T1]{fontenc}
\usepackage{amsmath}
\usepackage{amsfonts}
\usepackage{amssymb}
\usepackage[english]{babel}
\usepackage{tabularx}
\usepackage{mathtools}
\usepackage{adjustbox}
\usepackage{graphics}
\usepackage{nicefrac}       
\usepackage{microtype}      
\usepackage{lipsum}
\usepackage{lmodern}
\usepackage{hyperref}       
\usepackage{url}            
\usepackage{booktabs}
\DeclarePairedDelimiter\ceil{\lceil}{\rceil}
\DeclarePairedDelimiter\floor{\lfloor}{\rfloor}

\numberwithin{equation}{section}
\newtheorem{tw}{Theorem}[section]

\newtheorem{lm}[tw]{Lemma}
\newtheorem{uw}[tw]{Remark}
\newtheorem{df}[tw]{Definition}
\newtheorem{np}[tw]{Example}
\newtheorem{wn}[tw]{Corollary}

\newcommand{\Syl}{\text{Syl}}

\theoremstyle{remark}

\makeatletter
\let\@@pmod\pmod
\DeclareRobustCommand{\pmod}{\@ifstar\@pmods\@@pmod}
\def\@pmods#1{\mkern4mu({\operator@font mod}\mkern 6mu#1)}
\makeatother

\title{Discriminants of special quadrinomials}
\author{Krystian Gajdzica}
\address{Institute of Mathematics \\
	Faculty of Mathematics and Computer Science \\
	Jagiellonian University in Cracow
}
\email{krystian.gajdzica@im.uj.edu.pl}


\thanks{The research was partially supported by a grant of the National Science Centre (NCN), Poland, no. UMO-2019/34/E/ST1/00094.}

\begin{document}

\setlength{\parindent}{10mm}
\maketitle

\begin{abstract}
    Finding an effective formula for describing a discriminant of \linebreak a quadrinomial (a formula which can be easily computed for high values of degrees of quadrinomials) is a difficult problem. In 2018 Otake and Shaska \cite{ST} using advanced matrix operations found an explicit expression of $\Delta(x^n+t(x^2+ax+b))$. In this paper we focus on deriving similar results, taking advantage of alternative elementary approach, for quadrinomials of the form $x^n+ax^k+bx+c$, where $ k \in \{2,3,n-1\}$. Moreover, we make some notes about $\Delta(x^{2n}+ax^n+bx^l+c)$ such that $n>2l$.
\end{abstract}

\section{Introduction}
Let $p(x)=a_nx^n+a_{n-1}x^{n-1}+\ldots+a_1x+a_0$ be a polynomial of degree $n \geqslant 1$ over an arbitrary field $\mathbb{F}$. By a discriminant of $p$, we understand a product
$$ \Delta(p):= a_n^{2n-2}{\displaystyle \prod_{1 \leqslant i <j \leqslant n}(\xi_i-\xi_j)^2},$$
where $\xi_1,\ldots,\xi_n$ are roots of the polynomial $p$ in some extension of the field $\mathbb{F}$ (taken with multiplicities). Obviously, determining $\Delta(p)$ using the above formula is a difficult challenge in general. However, the discriminant is a special case of \linebreak a general object called resultant, which can be computed for each pair of non-trivial polynomials. More precisely, let $q(x) = b_mx^m+b_{m-1}x^{m-1}+\ldots+b_1x+b_0$ be a polynomial of degree $m \geqslant 0$ over the field $\mathbb{F}$, too. We define the resultant of polynomials $p$ and $q$ as the determinant of their Sylvester matrix $\Syl(p,q)$, see Definition $2.1$; and denote its by $R(p,q)$, i.e. $R(p,q)=\det(\normalfont{\Syl}(p,q))$. We are able to determine the discriminant of each non-constant polynomial using its relationship with the resultant, namely, if we take the initial polynomial $p$, then 
 $$\Delta(p)=(-1)^{\frac{n(n-1)}{2}}a_n^{-1}R(p,p').$$ 
 Hence, the discriminant might be interpreted as a product of a certain factor and Sylvester's matrix determinant. In spite of the ability to find discriminant of any polynomial, calculating the determinant is a computationally complex task even for large sparse matrices. Therefore, mathematicians deal with derivation more efficient formulae of the discriminant for some infinite families of polynomials. For instance, if we take a binomial $f(x)=x^n+a$, where $n>1$ and $a \in \mathbb{C}$, then 
 $$\Delta(f)=(-1)^\frac{n(n-1)}{2}n^na^{n-1},$$
 which can be easily computed even for high values of the parameter $n$. Similarly, we might describe the discriminant of a trinomial, namely, if $f(x)=x^n+ax^k+b$ is a polynomial over the field $\mathbb{C}$ such that $n \geqslant 3$, $ab \neq 0$, $n>k>0$ and $d=(n-k,k)$, then 
 $$\Delta(f)=(-1)^\frac{n(n-1)}{2}b^{k-1}\left(b^\frac{n-k}{d}n^\frac{n}{d}+(-1)^{\frac{n}{d}+1}a^\frac{n}{d}k^\frac{k}{d}(n-k)^\frac{n-k}{d}\right)^d;$$
 for more information, see \cite{GD}. The above identity is more intricate than the previous one, but it also can be easily calculated for large values of parameters $n$ and $k$.\linebreak A question arises whether there is an analogous description for the discriminant of a quadrinomial. Surprisingly, the only known result is the formula for $\Delta(x^n+t(x^2+ax+b))$, where $a,b,t \in \mathbb{C}$ and $n \geqslant 3$, see \cite{ST}. Specifically, in 2018 Otake and Shaska found the following identity
 \begin{equation*}
     \Delta(x^n+t(x^2+ax+b))=(-1)^{m_1}t^{n-1}\left((n-2)^{n-2}(a^2-4b)t^2+\gamma_ct-n^nb^{n-1}\right),
 \end{equation*}
 such that
 \begin{align*}
     \gamma_c=\sum_{k=0}^{m_0}(-1)^{n+k}n^k(n-1)^{n-2k-4}(n-2)^ka^{n-2k-4}b^k\cdot S_k,
 \end{align*}
 for $a\neq 0$ or $a=0$ and $n$ even, and $\gamma_c=0$ for $a=0$ and $n$ odd, and $m_0=\floor{(n-3)/2}$, $m_1=\ceil{(n-3)/2}$ and
 \begin{align*}
     S_k=&(n-1)^3\binom{n-k-3}{k}a^4+4n^2(n-2)\binom{n-k-4}{k}b^2\\
     &-\frac{n(n-1)\{5n^2-(6k+23)n+10k+24\}}{n-k-3}\binom{n-k-3}{k}a^2b.
 \end{align*}
 In order to obtain the result, they took advantage of so called Bezoutian of polynomials, that is a generalized concept of the resultant, and performed over ten pages of quite intense matrix computations. Nevertheless, their paper points out the method, which might be effective in determining the discriminant of a quadrinomial in general. 
 
 Afterward, in 2019 Jones \cite{LJ1} found an explicit formula for the discriminant of an irreducible polynomial of the form $f(x)=x^n+A(Bx+C)^m \in \mathbb{Z}[x],$ where $n \geqslant 3$ and $1 \leqslant  m<n.$ He proved that
 \begin{equation*}
     \Delta(f)=(-1)^\frac{n(n-1)}{2}C^{n(m-1)}A^{n-1}\left(n^nC^{n-m}+(-1)^{n+m}B^n(n-m)^{n-m}m^mA\right)
 \end{equation*}
 holds. In 2020 Jones further extended the foregoing result (see \cite{LJ2}) and obtained
 \begin{tw}
 Let $n$ and $k$ be integers with $1 \leqslant k < n.$ Let $f(x)=x^n+tg(x),$ where $t \in \mathbb{Z}$ and $g(x)=\sum_{i=0}^{k}a_ix^i \in \mathbb{Z}[x]$ with $a_0a_k \neq 0.$ Define $\hat{g}(x)=\sum_{i=0}^{k}a_i(n-i)x^i,$ and assume that 
 \begin{equation*}
     \hat{g}(x)=\prod_{i=1}^k(A_ix+B_i),
 \end{equation*}
 where the $A_ix+B_i \in \mathbb{Z}[x]$ are not necessarily distinct. If $f$ is irreducible, then
 \begin{equation*}
     \Delta(f)=(-1)^\frac{n(n+2k-1)}{2}t^{n-1}a_0^{-1}\prod_{i=1}^k\left((-B_i)^n+t\sum_{j=0}^ka_jA_i^{n-j}(-B_i)^j\right).
 \end{equation*}
 \end{tw}
 Furthermore, more recently he derived a discriminant formula for another one irreducible polynomial of the form $f(x)=x^{n-m}(x+k)^m+a,$ where $a,k,m,n \in \mathbb{Z},$ $ n \geqslant 3$ and $0<m<n$, namely,
 \begin{equation*}
     \Delta(f)=(-1)^\frac{n(n-1)}{2}a^{n-2}\left((-1)^{n-m}(n-m)^{n-m}k^nm^m+an^n\right),
 \end{equation*}
 for more details, see \cite{LJ3}.
 
 In order to expand a set of approaches to the discriminant of a quadrinomial problem, we use elementary algebraic resultant properties to derive effective formula for $\Delta(x^n+ax^k+bx+c)$, where $a,b,c \in \mathbb{C}$ and $k \in \{2,3,n-1\}$. The first impression suggests that our results can be directly concluded from the Jones's ones. Notwithstanding, we assume that a considered quadrinomial is neither irreducible over $\mathbb{Q}$ nor with coefficients in $\mathbb{Z}$.
 
 \begin{np}
 Let us take a polynomial $f(x)=x^{4n}-ix^3+ix+1 \in \mathbb{C}[x],$ where $n \geqslant 1$. Obviously, the quadrinomial $f$ does not satisfy assumptions of any of Jones's theorems. Therefore, we can not apply any of them to compute $\Delta(f)$, but we may use Theorem $5.3$.
 \end{np}
 
 At the end of the introduction, it needs to be highlighted that searching efficient formulae for discriminants of quadrinomials is not only art for art's sake, but such descriptions might be widely applied in algebraic number theory. For example, if the aforementioned polynomial $p$ belongs to $\mathbb{Z}[x]$ and is irreducible over $\mathbb{Q}$, $\mathbb{F}=\mathbb{Q}(\theta)$ and $p(\theta)=0$, then 
$\Delta(p)=[\mathbb{Z}_{\mathbb{F}}:\mathbb{Z}[\theta]]^2\Delta(\mathbb{F}),$
where $\mathbb{Z}_{\mathbb{F}}$ is the ring of integers of $\mathbb{F}$, then the first step to examine monogenicity of the polynomial $p$ (checking whether the equality $\Delta(p)=\Delta(\mathbb{F})$ holds) is to determine efficient formula for $\Delta(p)$ (see \cite{LJ1}). Another motivations of determining such identities come from \linebreak a bit more isolated situations and can be found, for instance, in \cite{CL, DS, G, ST}.
 
 This paper is organized as follows. In Sec. 2 we formally introduce the definitions of the resultant and the discriminant, their classical properties and connections between them. Sec. 3 is devoted to determining an efficient formula for $\Delta(x^n+ax^2+bx+c)$. In Sec. 4 we present similar consideration for the polynomial $x^n+ax^{n-1}+bx+c$. Sec. 5 concerns the discriminant of the quadrinomial $x^n+ax^3+bx+c$, which is the most difficult issue of this paper. Finally, in Sec. 6 we give some remarks and conclusions related to the discriminant of a quadrinomial and make some notes about $\Delta(x^{2n}+ax^n+bx^l+c)$, where $n>2l$.

\section{Basic results in resultants and discriminants}

In this section we recall standard definitions, facts and properties of the resultant and the discriminant. We omit all proofs, which are classical and can be easily found in \cite{SJ, TW}. Generally, (unless otherwise specified) we assume that each polynomial is considered over the field $\mathbb{C}$ of complex numbers. 
At the beginning, let us formally introduce the resultant and its characterization.

\subsection{Resultant} 
\phantom{That is the question...}

We begin our consideration by defining an object, which plays a central role in the sequel. Actually, there are a few equivalent definitions of the resultant, but we decide on the following classical one.
\begin{df}[Resultant]
Let $n,m$ be non-negative integers such that $n+m \geqslant 1$. Let $f(x)=a_nx^n+\ldots+a_0$ and $g(x)=b_mx^m+\ldots+b_0$ be polynomials of degrees $n$ and $m$, respectively, with coefficients in $\mathbb{C}$. The resultant $R(f,g)$ is an element of the field $\mathbb{C}$ defined as $\det(\normalfont{\Syl}(f,g))$, where $\normalfont{\Syl}(f,g)$ is given by 
\begin{equation*}
    \normalfont{\Syl}(f,g)=\begin{pmatrix}
  a_{n} & a_{n-1} & a_{n-2} & \cdots & 0 & 0 & 0 \\
  0 & a_{n} & a_{n-1} & \cdots & 0 & 0 & 0 \\
  \vdots  & \vdots  & \vdots & \ddots & \vdots & \vdots & \vdots  \\
  0 & 0 & 0 & \cdots & a_{1} & a_{0} & 0 \\
  0 & 0 & 0 & \cdots & a_{2} & a_{1} & a_{0} \\
  b_{m} & b_{m-1} & b_{m-2} & \cdots & 0 & 0 & 0 \\
  0 & b_{m} & b_{m-1} & \cdots & 0 & 0 & 0 \\
  \vdots  & \vdots  & \vdots & \ddots & \vdots & \vdots & \vdots  \\
  0 & 0 & 0 & \cdots & b_{1} & b_{0} & 0 \\
  0 & 0 & 0 & \cdots & b_{2} & b_{1} & b_{0} \\ 
 \end{pmatrix}\text{,} 
 \end{equation*}
 where first $m$ rows consist of the coefficients $a_n,\ldots,a_0$ shifted by $0,\ldots,m-1$ positions, respectively, and padded with zeros. Similarly, the last $n$ rows contain the coefficients $b_m,\ldots,b_0$ shifted by $0,\ldots,n-1$ positions, respectively, and padded with zeros.
\end{df}
The main application of the above object is to check whether two fixed polynomials over an arbitrary field possess a common root. From the standard properties of the determinant, we can automatically conclude a connection between $R(f,g)$ and $R(g,f)$.
\begin{wn}
If $f(x)=a_nx^n+\ldots+a_0$ and $g(x)=b_mx^m+\ldots+b_0,$ then
\begin{equation*}
R(f,g)=(-1)^{nm}R(g,f).
\end{equation*}
\end{wn}

In order to speed up the process of computing $R(f,g)$, we can take advantage of the commonly known Euclidean division of polynomials.   

\begin{tw}[Euclidean division of polynomials]
If $f,g \in \mathbb{C}[x]$ and $g \neq 0$, then there exists a unique pair $q,r \in \mathbb{C}[x]$ such that $f= qg+r$ and $\deg(r)<\deg(g)$.
\end{tw}

The foregoing theorem is directly related to an algebraic property of the resultant, which we systematically use in the sequel.

\begin{tw}
Let $f, g \in \mathbb{C}[x]$, $f(x)=a_nx^n+\ldots+a_0, g(x)=b_mx^m+\ldots+b_0$, $\deg(f)=n$ and $\deg(g)=m$. If $n \geqslant m, $ $h$ is a polynomial such that $\deg(h) \leqslant n-m$ and $\deg(f+hg)=k,$ then
\begin{equation*}
    R(f,g)=(-1)^{(n-k)m}b_m^{n-k}R(f+hg,g).
\end{equation*}
Similarly, if $n \leqslant m$, $h$ is a polynomial such that $\deg(h) \leqslant m-n$ and $\deg(g+hf)=k,$ then
\begin{equation*}
    R(f,g)=a_n^{m-k}R(f,g+hf).
\end{equation*}
\end{tw}

Some works present an alternative equivalent resultant definition connected with roots of considered polynomials, which we state in a distinct theorem style. 

\begin{tw}
Let $f(x)=a_nx^n+\ldots+a_0, g(x)=b_mx^m+\ldots+b_0$ be polynomials such that $\deg(f)=n$ and $\deg(g)=m$. Additionally, let us assume that $f(x)=a_n(x-\xi_1)\ldots(x-\xi_n)$ and $g(x)=b_m(x-\eta_1)\ldots(x-\eta_m).$ Then
\begin{equation}
    R(f,g)=a_n^mb_m^n\prod_{i=1}^{n}\prod_{j=1}^{m}(\xi_i-\eta_j).
\end{equation}
Equivalently, 
\begin{equation}
    R(f,g)=a_n^m\prod_{i=1}^{n}g(\xi_i)
\end{equation}
or
\begin{equation}
    R(f,g)=(-1)^{nm}b_m^n\prod_{j=1}^{m}f(\eta_j).
\end{equation}
\end{tw}




Theorem $2.5$ implies so called multiplicativity of the resultant with respect to each indeterminate. More precisely, if at least one of polynomials $f$ and $g$ can be expressed as a product, then we may split the problem of computing $R(f,g)$ into smaller pieces, which can be solved more easily.

\begin{tw}
If $f_1, f_2, g \in \mathbb{C}[x]$, then
\begin{equation*}
    R(f_1f_2,g)=R(f_1,g)R(f_2,g).
\end{equation*}
Symmetrically, if $g_1, g_2, f \in \mathbb{C}[x]$, then
\begin{equation*}
    R(f,g_1g_2)=R(f,g_1)R(f,g_2).
\end{equation*}
\end{tw}

Next, we recall the discriminant and its basic properties.

\subsection{Discriminant}

\begin{df}[Discriminant]
Let $f(x)=a_nx^n+\ldots+a_0$  be a polynomial of degree $\deg(f) \geqslant 1$ with coefficients in $\mathbb{C}$. We define the discriminant of $f$ as
\begin{equation}
    \Delta(f):= a_n^{2n-2}{\displaystyle \prod_{1 \leqslant i <j \leqslant n}(\xi_i-\xi_j)^2},
\end{equation}
where $\xi_1,\ldots,\xi_n$ are the roots (not necessarily distinct) of $f$.
\end{df}



Comparing the equation $(2.4)$ and Theorem $2.5$ one can easily deduce the following.

\begin{tw}
If $f(x)=a_nx^n+\ldots+a_0$ is a polynomial such taht $\deg(f) \geqslant 1$, then 
\begin{equation}
    \Delta(f)=(-1)^\frac{n(n-1)}{2}a_n^{-1}R(f,f').
\end{equation}
\end{tw}
Finally, we remind both the reciprocal polynomial and a well-known theorem, from which we further deduce effective formulae for $\Delta(x^n+ax^{n-1}+bx^l+c)$, where $abc\neq0$ and $l\in\{n-3,n-2\}$.

\begin{df}[Reciprocal polynomial]
Let $f(x)=a_nx^n+\ldots+a_0\in\mathbb{C}[x]$. The reciprocal polynomial of $f$ is defined as
\begin{equation*}
    f^\star(x):=x^nf(1/x)=a_0x^n+a_1x^{n-1}+\ldots+a_n.
\end{equation*}
\end{df}

\begin{tw}
If $f(x)=a_nx^n+\ldots+a_0$ is a polynomial such that $a_na_0 \neq 0$, then
\begin{equation*}
    \Delta(f)=\Delta(f^\star).
\end{equation*}
\end{tw}

Obviously, there are many other facts on the discriminant. However, we pay attention only to those which are significant in our main investigation.

\section{The discriminant of the quadrinomial $f(x)=x^n+ax^2+bx+c$}

Now, we focus on the discriminant considered by Otake and Shaska in \cite{ST}. However, in contrast to their Bezoutian approach, we show an elementary and direct way to derive the explicit formula for $\Delta(x^n+ax^2+bx+c)$.

\begin{tw}
The discriminant of a polynomial $f(x)=x^n+ax^2+bx+c$ such that $abc \neq 0$ and $n>3$ can be expressed as
\begin{align*}
     \Delta(f)=&(-1)^\frac{n(n-1)}{2}n^{-1}(n-2)^{n-1}a^{n-1}\\&\cdot\Bigg[n^2e_2^{n-1}+4a^2e_2+2abe_1+b^2+\left(1+(-1)^{n-1}\right)bn\left(\left(\frac{e_1}{2}\right)^2-e_2\right)^\frac{n-1}{2}\\
     &\phantom{\cdot}+n\sum_{i=0}^{\floor{\frac{n-2}{2}}}\left(4ae_2+\frac{n-1}{n-1-2i}be_1\right)\binom{n-2}{2i}\left(\frac{e_1}{2}\right)^{n-2-2i}\left(\left(\frac{e_1}{2}\right)^2-e_2\right)^{i}\Bigg]
     ,
\end{align*}
where $e_1=-\frac{(n-1)b}{(n-2)a}$ and $e_2=\frac{nc}{(n-2)a}.$
\end{tw}

\begin{proof}
If $f(x)=x^n+ax^2+bx+c$, $abc \neq 0$ and $n>3$, then, clearly, we have $f'(x)=nx^{n-1}+2ax+b$. Theorem $2.8$ and Corollary $2.2$ provide that
$$ \Delta(f)=(-1)^\frac{n(n-1)}{2}R(f,f')=(-1)^\frac{n(n-1)}{2}R(f',f).$$
Now, we apply Theorem $2.3$ in order to find the remainder $r$ on dividing $f$ by $f'$, namely,
$$x^n+ax^2+bx+c = \frac{1}{n}x(nx^{n-1}+2ax+b) + \frac{a(n-2)}{n}x^2+\frac{b(n-1)}{n}x+c.$$
Thus, 
$$r(x):=\frac{a(n-2)}{n}x^2+\frac{b(n-1)}{n}x+c,$$
and Theorem $2.4$ asserts that
$$R(f',f)=n^{n-2}R(f',r).$$
Next, we can easily determine the roots of the quadratic polynomial $r$, which are given by
$$z_1=-\frac{(n-1)b}{2(n-2)a}+\sqrt{\left(\frac{(n-1)b}{2(n-2)a}\right)^2-\frac{nc}{(n-2)a}}$$
and
$$z_2=-\frac{(n-1)b}{2(n-2)a}-\sqrt{\left(\frac{(n-1)b}{2(n-2)a}\right)^2-\frac{nc}{(n-2)a}}.$$
Let us define:  $e_1:=z_1+z_2=-\frac{(n-1)b}{(n-2)a}$ and $e_2:=z_1z_2=\frac{nc}{(n-2)a}.$ We use Theorem $2.5$, especially the equality $(2.3),$ and compute $R(f',r)$, as follows
\begin{align*}
     R(f',r)=& (-1)^{2(n-1)}\left(\frac{(n-2)a}{n}\right)^{n-1}\prod_{j=1}^{2}f'(z_j)\\
     =& \left(\frac{(n-2)a}{n}\right)^{n-1}\left(nz_1^{n-1}+2az_1+b\right)\left(nz_2^{n-1}+2az_2+b\right)\\
     =& \left(\frac{(n-2)a}{n}\right)^{n-1}\big(n^2(z_1z_2)^{n-1}+4a^2z_1z_2+2ab(z_1+z_2)\\
     & +b^2+2naz_1z_2(z_1^{n-2}+z_2^{n-2})+nb(z_1^{n-1}+z_2^{n-1})\big).
\end{align*}
Moreover, for each $s\in\mathbb{N}_+$ we have
\begin{align*}
     z_1^s&=\left(-\frac{(n-1)b}{2(n-2)a}+\sqrt{\left(\frac{(n-1)b}{2(n-2)a}\right)^2-\frac{nc}{(n-2)a}}\right)^s \\
     &=\left(\frac{e_1}{2}+\sqrt{\left(\frac{e_1}{2}\right)^2-e_2}\right)^s=\sum_{i=0}^s\binom{s}{i}\left(\frac{e_1}{2}\right)^{s-i}\left(\left(\frac{e_1}{2}\right)^2-e_2\right)^{\frac{i}{2}}
\end{align*}
and
\begin{align*}
     z_2^s&=\left(-\frac{(n-1)b}{2(n-2)a}-\sqrt{\left(\frac{(n-1)b}{2(n-2)a}\right)^2-\frac{nc}{(n-2)a}}\right)^s \\
     &=\left(\frac{e_1}{2}-\sqrt{\left(\frac{e_1}{2}\right)^2-e_2}\right)^s=\sum_{i=0}^s(-1)^{i}\binom{s}{i}\left(\frac{e_1}{2}\right)^{s-i}\left(\left(\frac{e_1}{2}\right)^2-e_2\right)^{\frac{i}{2}}.
\end{align*}
Therefore, the sum $z_1^s+z_2^s$ can be written as
\begin{align*}
     z_1^s+z_2^s&=\sum_{i=0}^s\left(1+(-1)^i\right)\binom{s}{i}\left(\frac{e_1}{2}\right)^{s-i}\left(\left(\frac{e_1}{2}\right)^2-e_2\right)^{\frac{i}{2}}\\
     &=2\sum_{i=0}^{\floor{\frac{s}{2}}}\binom{s}{2i}\left(\frac{e_1}{2}\right)^{s-2i}\left(\left(\frac{e_1}{2}\right)^2-e_2\right)^{i}.
\end{align*}
It is worth mentioning that the square root terms in the above formula disappear. In particular, it means that the sum $z_1^s+z_2^s$ is a polynomial in variables $b/a$ and $c/a$ with coefficients in the field $\mathbb{Q}$. Consequently, the formula for $R(f',r)$ takes the form
\begin{align*}
     R(f',r)&=\left(\frac{(n-2)a}{n}\right)^{n-1}\Bigg[n^2e_2^{n-1}+4a^2e_2+2abe_1+b^2\\
     &\phantom{=\left(\frac{(n-2)a}{n}\right)^{n-1}\Bigg[}+\left(1+(-1)^{n-1}\right)bn\left(\left(\frac{e_1}{2}\right)^2-e_2\right)^\frac{n-1}{2}\\
     &+n\sum_{i=0}^{\floor{\frac{n-2}{2}}}\left(4ae_2+\frac{n-1}{n-1-2i}be_1\right)\binom{n-2}{2i}\left(\frac{e_1}{2}\right)^{n-2-2i}\left(\left(\frac{e_1}{2}\right)^2-e_2\right)^{i}\Bigg].
\end{align*}
Finally, the discriminant of the quadrinomial $x^n+ax^2+bx+c$ can be represented as
\begin{align*}
     \Delta(f)=&(-1)^\frac{n(n-1)}{2}n^{n-2}\left(\frac{(n-2)a}{n}\right)^{n-1}\\
     &\cdot\Bigg[n^2e_2^{n-1}+4a^2e_2+2abe_1+b^2+\left(1+(-1)^{n-1}\right)bn\left(\left(\frac{e_1}{2}\right)^2-e_2\right)^\frac{n-1}{2}\\
     &\phantom{\cdot}+n\sum_{i=0}^{\floor{\frac{n-2}{2}}}\left(4ae_2+\frac{n-1}{n-1-2i}be_1\right)\binom{n-2}{2i}\left(\frac{e_1}{2}\right)^{n-2-2i}\left(\left(\frac{e_1}{2}\right)^2-e_2\right)^{i}\Bigg].
\end{align*}
If we reduce appropriate expressions in the above formula, we obtain the statement of our theorem.
\end{proof}

\begin{uw}
\rm{Let us point out that the assumption of $n>3$ in Theorem $3.1$ maintains that $\deg(r)=2$. Also, we can consider the case in which $n=3$. Then the remainder $r$ on dividing $f$ by $f'$ is a linear polynomial. If we utilize the approach of the above proof in this case, we get
$$\Delta(f)=-4a^3c+a^2b^2-4b^3+18abc-27c^2.$$}
\end{uw}
If we connect the foregoing result and Theorem $2.10$, we will be able to present an effective discriminant formula for another one quadrinomial.

\begin{wn}
The discriminant of a polynomial $f(x)=x^n+ax^{n-1}+bx^{n-2}+c$ such that $abc \neq 0$ and $n>3$ can be written as
\begin{align*}
     \Delta(f)=&(-1)^\frac{n(n-1)}{2}n^{-1}\left((n-2)bc\right)^{n-1}\\&\cdot\Bigg[n^2e_2^{n-1}+\frac{4b^2e_2+2abe_1+a^2}{c^2}+\left(1+(-1)^{n-1}\right)\frac{an}{c}\left(\left(\frac{e_1}{2}\right)^2-e_2\right)^\frac{n-1}{2}\\
     &\phantom{\cdot}+\frac{n}{c}\sum_{i=0}^{\floor{\frac{n-2}{2}}}\left(4be_2+\frac{n-1}{n-1-2i}ae_1\right)\binom{n-2}{2i}\left(\frac{e_1}{2}\right)^{n-2-2i}\left(\left(\frac{e_1}{2}\right)^2-e_2\right)^{i}\Bigg]
     ,
\end{align*}
where $e_1=-\frac{(n-1)a}{(n-2)b}$ and $e_2=\frac{n}{(n-2)b}.$
\end{wn}

\begin{proof}
If $f(x)=x^n+ax^{n-1}+bx^{n-2}+c$, $abc \neq 0$ and $n>3$, then the reciprocal polynomial $f^\star$ is given by $f^\star(x)=cx^n+bx^2+ax+1$ and Theorem $2.10$ asserts that $\Delta(f)=\Delta(f^\star)$. From Definition $2.7$ we have $$\Delta(cx^n+bx^2+ax+1)=c^{2n-2}\Delta\left(x^n+\frac{b}{c}x^2+\frac{a}{c}x+\frac{1}{c}\right).$$ In order to find the desired formula, we use Theorem $3.1$ and compute the last discriminant in the above equation.
\end{proof}

\section{The discriminant of the quadrinomial $f(x)=x^n+ax^{n-1}+bx+c$}

In this section we obtain the expression for the discriminant of the quadrinomial $f(x)=x^n+ax^{n-1}+bx+c$. In the proof we use the same approach as in our computations of $\Delta(x^n+ax^2+bx+c)$.
\begin{tw}
The explicit discriminant formula of a polynomial $f(x)=x^n+ax^{n-1}+bx+c$ such that $abc\neq0$ and $n>4$ is given by
\begin{align*}
     \Delta(f)&=(-1)^\frac{(n+2)(n-1)}{2}n^4(n-1)^{n-3}a^{-2}b^{n-2}\\
     &\cdot\Bigg[\left(\frac{(n-1)a^2}{n^2}\right)^2e_2^{n-2}+\left(\frac{(n-1)b}{n}\right)^2e_2+\frac{(n-1)b}{n}\left(c-\frac{ab}{n^2}\right)e_1\\
     &\phantom{\cdot\Bigg[}+\left(c-\frac{ab}{n^2}\right)^2-\left(1+(-1)^n\right)\left(\frac{(n-1)a^2}{n^2}\right)\left(c-\frac{ab}{n^2}\right)\left(\left(\frac{e_1}{2}\right)^2-e_2\right)^\frac{n-2}{2}\\
     &\phantom{\cdot\Bigg[}-\left(\frac{(n-1)a^2}{n^2}\right)\sum_{i=0}^{\floor{\frac{n-3}{2}}}\left(2\frac{(n-1)b}{n}e_2+\left(c-\frac{ab}{n^2}\right)\frac{n-2}{n-2-2i}e_1\right)\\
     &\phantom{\cdot\Bigg[-\left(\frac{(n-1)a^2}{n^2}\right)\sum_{i=0}^{\floor{\frac{n-3}{2}}}}\cdot\binom{n-3}{2i}\left(\frac{e_1}{2}\right)^{n-3-2i}\left(\left(\frac{e_1}{2}\right)^2-e_2\right)^i\Bigg],
\end{align*}
where $e_1=-\frac{(n-2)ab+cn}{b(n-1)}$ and $e_2=\frac{ac}{b}.$
\end{tw}

\begin{proof}
Let $f(x)=x^n+ax^{n-1}+bx+c$ be a polynomial such that $abc\neq0$ and $n>4.$ Then $f'(x)=nx^{n-1}+(n-1)ax^{n-2}+b$ and
$$x^n+ax^{n-1}+bx+c=\left(\frac{1}{n}x+\frac{a}{n^2}\right)\left(nx^{n-1}+(n-1)ax^{n-2}+b\right)-r(x),$$
where
$$r(x)=\frac{n-1}{n^2}a^2x^{n-2}-\frac{n-1}{n}bx-c+\frac{ab}{n^2}.$$
Furthermore, Theorem $2.8$, Corollary $2.2$ and Theorem $2.4$ provide that
\begin{align*}
    \Delta(f)&=(-1)^\frac{n(n-1)}{2}R(f',f)=(-1)^\frac{n(n-1)}{2}n^2R(f',-r)\\
    &=(-1)^\frac{n(n-1)}{2}n^2(-1)^{n-1}R(f',r)=(-1)^\frac{(n+2)(n-1)}{2}n^2R(r,f').
\end{align*}
Let us note that
$$f'(x)=\left(\frac{n^3}{(n-1)a^2}x+\frac{n^2}{a}\right)\left(\frac{n-1}{n^2}a^2x^{n-2}-\frac{n-1}{n}bx-c+\frac{ab}{n^2}\right)+\Tilde{r}(x),$$
where
$$\Tilde{r}(x)=n^2\left(\frac{b}{a^2}x^2+\frac{(n-2)ab+cn}{(n-1)a^2}x+\frac{c}{a}\right).$$
Once again we use Theorem $2.4$ and get
$$R(r,f')=\left(\frac{(n-1)a^2}{n^2}\right)^{n-3}R(r,\Tilde{r}).$$
The roots of quadratic polynomial $\Tilde{r}$ are given by
$$z_1=-\frac{(n-2)ab+cn}{2b(n-1)}+\sqrt{\left(\frac{(n-2)ab+cn}{2b(n-1)}\right)^2-\frac{ac}{b}},$$
$$z_2=-\frac{(n-2)ab+cn}{2b(n-1)}-\sqrt{\left(\frac{(n-2)ab+cn}{2b(n-1)}\right)^2-\frac{ac}{b}}.$$
Similarly to the previous proof, let us denote the sum and the product of the above expressions as $e_1:=z_1+z_2=-\frac{(n-2)ab+cn}{b(n-1)}, e_2:=z_1z_2=\frac{ac}{b}.$
If we connect Theorem $2.5$ with the computations carried out in the prior section, we obtain
\begin{align*}
     &R(r,\Tilde{r})=(-1)^{2(n-2)}\left(\frac{n^2b}{a^2}\right)^{n-2}\prod_{j=1}^{2}r(z_j)=\left(\frac{n^2b}{a^2}\right)^{n-2}\\
     &\cdot\left(\frac{n-1}{n^2}a^2z_1^{n-2}-\frac{n-1}{n}bz_1-c+\frac{ab}{n^2}\right)\left(\frac{n-1}{n^2}a^2z_2^{n-2}-\frac{n-1}{n}bz_2-c+\frac{ab}{n^2}\right)\\
     &=\left(\frac{n^2b}{a^2}\right)^{n-2}\Bigg[\left(\frac{(n-1)a^2}{n^2}\right)^2e_2^{n-2}+\left(\frac{(n-1)b}{n}\right)^2e_2+\frac{(n-1)b}{n}\left(c-\frac{ab}{n^2}\right)e_1\\
     &+\left(c-\frac{ab}{n^2}\right)^2-\left(1+(-1)^n\right)\left(\frac{(n-1)a^2}{n^2}\right)\left(c-\frac{ab}{n^2}\right)\left(\left(\frac{e_1}{2}\right)^2-e_2\right)^\frac{n-2}{2}\\
     &-\left(\frac{(n-1)a^2}{n^2}\right)\sum_{i=0}^{\floor{\frac{n-3}{2}}}\left(2\frac{(n-1)b}{n}e_2+\left(c-\frac{ab}{n^2}\right)\frac{n-2}{n-2-2i}e_1\right)\\
     &\phantom{-\left(\frac{(n-1)a^2}{n^2}\right)\sum_{i=0}^{\floor{\frac{n-3}{2}}}}\cdot\binom{n-3}{2i}\left(\frac{e_1}{2}\right)^{n-3-2i}\left(\left(\frac{e_1}{2}\right)^2-e_2\right)^i\Bigg].
\end{align*}
Consequently, the explicit formula of the discriminant $\Delta(f)$ takes the form
\begin{align*}
     &\Delta(f)=(-1)^\frac{(n+2)(n-1)}{2}n^2\left(\frac{(n-1)a^2}{n^2}\right)^{n-3}\left(\frac{n^2b}{a^2}\right)^{n-2}\\
     &\cdot\Bigg[\left(\frac{(n-1)a^2}{n^2}\right)^2e_2^{n-2}+\left(\frac{(n-1)b}{n}\right)^2e_2+\frac{(n-1)b}{n}\left(c-\frac{ab}{n^2}\right)e_1+\left(c-\frac{ab}{n^2}\right)^2\\
     &-\left(1+(-1)^n\right)\left(\frac{(n-1)a^2}{n^2}\right)\left(c-\frac{ab}{n^2}\right)\left(\left(\frac{e_1}{2}\right)^2-e_2\right)^\frac{n-2}{2}\\
     &-\left(\frac{(n-1)a^2}{n^2}\right)\sum_{i=0}^{\floor{\frac{n-2}{2}}}\left(2\frac{(n-1)b}{n}e_2+\left(c-\frac{ab}{n^2}\right)\frac{n-2}{n-2-2i}e_1\right)\\
     &\phantom{-\left(\frac{(n-1)a^2}{n^2}\right)\sum_{i=0}^{\floor{\frac{n-2}{2}}}}\cdot\binom{n-3}{2i}\left(\frac{e_1}{2}\right)^{n-3-2i}\left(\left(\frac{e_1}{2}\right)^2-e_2\right)^i\Bigg].
\end{align*}
The appropriate reduction of the above formula ends the proof.
\end{proof}

The foregoing proof consists of one additional polynomials division according to the one described in Sec. 3. However, as we see, the general concept remains the same in both cases.

\begin{uw}
\rm{If we assume in Theorem $4.1$ that $n=3$, we obtain a result from Remark $3.2.$ On the other hand, if we consider the polynomial $f(x)=x^4+ax^3+bx+c$, then the remainder $r$ on dividing $f$ by $f'$ is \text{a quadratic} polynomial. We can take advantage of techniques from proof of Theorem $3.1$ and check that
$$\Delta(f)=-4a^3b^3-27a^4c^2-6a^2b^2c-27b^4-192abc^2+256c^3.$$}
\end{uw}

\section{The discriminant of the quadrinomial $f(x)=x^n+ax^3+bx+c$}

The next part concerns the quadrinomial $x^n+ax^3+bx+c$ and demands more work than the previous ones. In order to derive the explicit formula of $\Delta(f)$ we take advantage of the generalized remainder theorem (for more details, see \cite{WW}).

\begin{tw}
If $p(x)=a_nx^n+\ldots+a_0$ and $q(x)=b_mx^m-b_{m-1}x^{m-1}-\ldots-b_0$ are polynomials over the field $\mathbb{C}$ such that $a_nb_m \neq 0$ and $n \geqslant m,$ then the remainder on dividing $p(x)$ by $q(x)$ is
\begin{equation}
    r(x)=\sum_{k=0}^{m-1}\left(a_k+\frac{1}{b_m}\sum_{i=0}^{k}b_i\sum_{\nu=0}^{n-m-k+i}t_{n-m-k+i+1-\nu}a_{n-\nu}\right)x^k,
\end{equation}
where $(t_r)$ is the linear recurrent sequence, defined as: 
\begin{equation}
     t_r = \begin{cases} 
          \frac{1}{b_m}, & r=1, \\
          \frac{1}{b_m}\sum_{i=1}^{r-1}b_{m-i}t_{r-i}, & r>1.
       \end{cases}
\end{equation}
\end{tw}

First, we compute an explicit formula of the sequence $(t_r)$ for a special pair of polynomials $p$ and $q$. 

\begin{lm}
If $p(x)=a_{n-1}x^{n-1}+a_2x^2+a_0$, $q(x)=b_3x^3-b_1x-b_0$ are polynomials such that $a_{n-1}b_3\neq0$ and $n\geqslant 5$, then the solution of recurrence relation $(5.2)$ is given by
\begin{equation}
     t_r = \begin{cases} 
          \left(\frac{1}{b_3}\right)^{\ceil{\frac{r}{2}}}\sum_{k=0}^{\floor{\frac{r}{6}}}\binom{\floor{\frac{r}{2}}-k}{\floor{\frac{r}{2}}-3k}b_0^{2k}b_1^{\floor{\frac{r}{2}}-3k}b_3^k, & 2\nmid r, \\
          \left(\frac{1}{b_3}\right)^{\frac{r}{2}}\sum_{k=1}^{\floor{\frac{r+2}{6}}}\binom{\frac{r}{2}-k}{\frac{r}{2}-3k+1}b_0^{2k-1}b_1^{\frac{r}{2}-3k+1}b_3^{k-1}, & 2|r.
       \end{cases}
\end{equation}
\end{lm}

\begin{proof}
The lemma can be proven by induction on $r$. It is clear for $r=1$, because $t_1=\frac{1}{b_3}.$ If $r=2$, then the equation $(5.3)$ gives us $t_2=\frac{0}{b_3}=0.$ Similarly, the relation $(5.2)$ maintains that $t_2=\frac{1}{b_3}b_2t_1=0$, since $b_2=0$ and we get the equality $(5.3)$ for $r=2$. Let us also consider the case when $r=3$. Then the formula $(5.3)$ implies that $t_3=\frac{b_1}{b_3^2}$. On the other hand, $(5.2)$ implies that $t_3=\frac{1}{b_3}\left(b_2t_2+b_1t_1\right)=\frac{b_1}{b_3^2}$. We assume that the formula $(5.3)$ holds for each $r=1,2,\ldots,k-1$, and show that it also holds for $r=k$. We need to examine four cases: $k\equiv 0,2 \pmod{6}$, $k\equiv 1 \pmod{6}$, $k\equiv 3,5 \pmod{6}$ and $k\equiv 4 \pmod{6}.$ Because in all possibilities the reasoning goes exactly in the same way we present only the situation in which $k\equiv 2 \pmod{6}$. At the beginning, let us observe that the formula $(5.2)$ takes the form $t_r=\frac{1}{b_3}\left(b_1t_{r-2}+b_0t_{r-3}\right)$ for $r>3$. Consequently, we obtain $t_k=\frac{1}{b_3}\left(b_1t_{k-2}+b_0t_{k-3}\right)$. This equality together with the induction assumption and the form of the number $k$ provide that
\begin{align*}
    t_k=\frac{1}{b_3}&\Biggl(b_1\left(\frac{1}{b_3}\right)^{\frac{k-2}{2}}\sum_{i=1}^{\floor{\frac{k}{6}}}\binom{\frac{k-2}{2}-i}{\frac{k-2}{2}-3i+1}b_0^{2i-1}b_1^{\frac{k-2}{2}-3i+1}b_3^{i-1} \\
    &+b_0\left(\frac{1}{b_3}\right)^{\ceil{\frac{k-3}{2}}}\sum_{i=0}^{\floor{\frac{k-3}{6}}}\binom{\floor{\frac{k-3}{2}}-i}{\floor{\frac{k-3}{2}}-3i}b_0^{2i}b_1^{\floor{\frac{k-3}{2}}-3i}b_3^i\Biggr).
\end{align*}
Let us suppose that $\floor{\frac{k}{6}}=s$ for some $s \in \mathbb{N}$. Then, we have 
\begin{align*}
    t_k=\frac{1}{b_3}&\Biggl(\left(\frac{1}{b_3}\right)^{\frac{k-2}{2}}\sum_{i=1}^{s}\binom{\frac{k-2}{2}-i}{\frac{k-2}{2}-3i+1}b_0^{2i-1}b_1^{\frac{k-2}{2}-3i+2}b_3^{i-1} \\
    &+\left(\frac{1}{b_3}\right)^{\frac{k-2}{2}}\sum_{i=0}^{s-1}\binom{\floor{\frac{k-3}{2}}-i}{\floor{\frac{k-3}{2}}-3i}b_0^{2i+1}b_1^{\floor{\frac{k-3}{2}}-3i}b_3^i\Biggr).
\end{align*}
Now, let us interchange the index in the second sum and take a common factor in front of the main parentheses
\begin{align*}
    t_k=\left(\frac{1}{b_3}\right)^{\frac{k}{2}}&\Biggl(\sum_{i=1}^{s}\binom{\frac{k-2}{2}-i}{\frac{k-2}{2}-3i+1}b_0^{2i-1}b_1^{\frac{k-2}{2}-3i+2}b_3^{i-1} \\
    &+\sum_{i=1}^{s}\binom{\floor{\frac{k-3}{2}}-i+1}{\floor{\frac{k-3}{2}}-3i+3}b_0^{2i-1}b_1^{\floor{\frac{k-3}{2}}-3i+3}b_3^{i-1}\Biggr).
\end{align*}
Additionally, in our case we have that $\floor{\frac{k-3}{2}}=\frac{k-2}{2}-1$. Because, $\binom{n}{k}=\binom{n-1}{k-1}+\binom{n-1}{k}$ for $n,k \in \mathbb{N}$ and  $\floor{\frac{k+2}{6}}=\floor{\frac{k}{6}}=s$, we conclude that
$$ t_k= \left(\frac{1}{b_3}\right)^{\frac{k}{2}}\sum_{i=1}^{\floor{\frac{k+2}{6}}}\binom{\frac{k}{2}-i}{\frac{k}{2}-3i+1}b_0^{2i-1}b_1^{\frac{k}{2}-3i+1}b_3^{i-1},$$
which agrees with the equation $(5.3)$. Similar computations in remaining possible cases complement the proof.
\end{proof}

If we use Lemma $5.2$ we may go through the most elaborate part related to derivating an explicit formula for $\Delta(x^n+ax^3+bx+c)$. In fact, the primary result of this section is more complex than the preceding ones.

\begin{tw}
The discriminant of a quadrinomial $f(x)=x^n+ax^3+bx+c$ such that $abc\neq 0$ and $n>4$ is given by
\begin{align*}
    \Delta(f)=&(-1)^{\frac{(n-1)(n+2)}{2}}a^{n-3}(n-3)^{n-3}(3a+nt_{n-2})^3\\
    &\Bigg(\left(\frac{a(n-3)}{n}\right)^2e_2^3+\frac{ab(n-1)(n-3)}{n^2}(e_1^2-2e_2)e_2\\
    &+\frac{ac(n-3)}{n}(e_1^3-3e_1e_2) +\left(\frac{b(n-1)}{n}\right)^2e_2+\frac{bc(n-1)}{n}e_1+c^2\Bigg),
\end{align*}
where $e_1=-\frac{nt_{n-1}}{3a+nt_{n-2}}, e_2= \frac{ab(n-3)-cn^2t_{n-3}}{a(n-3)(3a+nt_{n-2})}$ and the sequence $(t_r)$ is expressed as
\begin{align*}
t_r = \begin{cases}
          \left(\frac{n}{a(n-3)}\right)^{\ceil{\frac{r}{2}}}\sum_{k=0}^{\floor{\frac{r}{6}}}\binom{\floor{\frac{r}{2}}-k}{\floor{\frac{r}{2}}-3k}c^{2k}\left(\frac{b(n-1)}{n}\right)^{\floor{\frac{r}{2}}-3k}\left(\frac{a(n-3)}{n}\right)^k, & 2\nmid r, \\
          \left(\frac{n}{a(n-3)}\right)^{\frac{r}{2}}\sum_{k=1}^{\floor{\frac{r+2}{6}}}\binom{\frac{r}{2}-k}{\frac{r}{2}-3k+1}c^{2k-1}\left(\frac{b(n-1)}{n}\right)^{\frac{r}{2}-3k+1}\left(\frac{a(n-3)}{n}\right)^{k-1}, & 2|r.
       \end{cases}
\end{align*}
\end{tw}

\begin{proof}
Let the assumptions of the theorem hold. Then $f'(x)=nx^{n-1}+3ax^2+b$ and Theorem $2.8$ implies that
$$\Delta(f)=(-1)^\frac{n(n-1)}{2}R(f,f').$$
Due to Euclidean division of polynomials we compute the remainder $r$ on dividing $f$ by $f'$
$$x^n+ax^3+bx+c=\frac{1}{n}x\left(nx^{n-1}+3ax^2+b\right)+r(x),$$
where
$$r(x)=\frac{a(n-3)}{n}x^3+\frac{b(n-1)}{n}x+c.$$
Corollary $2.2$ and Theorem $2.4$ assert that
$$R(f,f')=R(f',f)=n^{n-3}R(f',r).$$
In order to specify the remainder $\Tilde{r}$ on dividing $f'$ by $r$, we use Theorem $5.1$. Let us observe that in this case $a_i$ are corresponding coefficients of polynomial $f'$. On the other hand, we put $b_3=\frac{a(n-3)}{n}, b_2=0, b_1=-\frac{b(n-1)}{n}$ and $b_0=-c,$ because of the assumptions of Theorem $5.1$. By a few calculations and the fact that the sequence $(t_l)$ satisfies $t_3=\frac{b_1}{b_3}t_1$ and $t_s=\frac{1}{b_3}(b_1t_{s-2}+b_0t_{s-3})$ for $s>3$, we can conclude that
\begin{align*}
   \Tilde{r}(x)&=\sum_{k=0}^{2}\left(a_k+\frac{1}{b_3}\sum_{i=0}^{k}b_i\sum_{\nu=0}^{n-4-k+i}t_{n-3-k+i-\nu}a_{n-1-\nu}\right)x^k \\
 &=(a_2+a_{n-1}t_{n-2})x^2+a_{n-1}t_{n-1}x+a_0+\frac{b_0}{b_3}t_{n-3}a_{n-1}, 
\end{align*}
where the sequence $(t_l)$ can be expressed by the equation $(5.3)$. Using Corollary $2.2$ and Theorem $2.4$ again we get that
\begin{align*}
    R(f',r)&=(-1)^{n-1}R(r,f')=(-1)^{n-1}\left(\frac{a(n-3)}{n}\right)^{n-3}R(r,\Tilde{r}).
\end{align*}
The roots of the polynomial $\Tilde{r}$ are given by
\begin{align*}
    \xi_1&=\frac{-a_{n-1}t_{n-1}+\sqrt{(a_{n-1}t_{n-1})^2-4(a_2+a_{n-1}t_{n-2})(a_0+\frac{b_0}{b_3}t_{n-3}a_{n-1})}}{2(a_2+a_{n-1}t_{n-2})}\\
    &=\frac{-nt_{n-1}+\sqrt{(nt_{n-1})^2-4(3a+nt_{n-2})(b-\frac{cn^2}{a(n-3)}t_{n-3})}}{2(3a+nt_{n-2})}
\end{align*}
and
\begin{align*}
    \xi_2&=\frac{-a_{n-1}t_{n-1}-\sqrt{(a_{n-1}t_{n-1})^2-4(a_2+a_{n-1}t_{n-2})(a_0+\frac{b_0}{b_3}t_{n-3}a_{n-1})}}{2(a_2+a_{n-1}t_{n-2})} \\
    &=\frac{-nt_{n-1}-\sqrt{(nt_{n-1})^2-4(3a+nt_{n-2})(b-\frac{cn^2}{a(n-3)}t_{n-3})}}{2(3a+nt_{n-2})}.
\end{align*}
Now, we use $(2.3)$ to describe $R(r,\Tilde{r})$ explicitly, as
$$R(r,\Tilde{r})=(a_2+a_{n-1}t_{n-2})^3\prod_{i=1}^2\left(\frac{a(n-3)}{n}\xi_i^3+\frac{b(n-1)}{n}\xi_i+c\right).$$
Finally, let us combine all found formulae in order to obtain
\begin{align*}
    \Delta(f)=&(-1)^{\frac{(n-1)(n+2)}{2}}a^{n-3}(n-3)^{n-3}(3a+nt_{n-2})^3\\
    &\prod_{i=1}^2\left(\frac{a(n-3)}{n}\xi_i^3+\frac{b(n-1)}{n}\xi_i+c\right).
\end{align*}
Once again we introduce both the sum and the product of the roots of $\Tilde{r}$, that are $e_1:=\xi_1+\xi_2=-\frac{nt_{n-1}}{3a+nt_{n-2}}$ and $e_2:=\xi_1\xi_2=\frac{ab(n-3)-cn^2t_{n-3}}{a(n-3)(3a+nt_{n-2})}.$
Eventually, on the base of computation in the proof of Theorem $3.1$ we conclude
\begin{align*}
    \Delta(f)=&(-1)^{\frac{(n-1)(n+2)}{2}}a^{n-3}(n-3)^{n-3}(3a+nt_{n-2})^3\\
    &\Bigg(\left(\frac{a(n-3)}{n}\right)^2e_2^3+\frac{ab(n-1)(n-3)}{n^2}(e_1^2-2e_2)e_2\\
    &+\frac{ac(n-3)}{n}(e_1^3-3e_1e_2) +\left(\frac{b(n-1)}{n}\right)^2e_2+\frac{bc(n-1)}{n}e_1+c^2\Bigg),
\end{align*}
which was to be shown.
\end{proof}

It is worth stressing that at the first glance the obtained formula seems to be shorter and easier than the previous ones. However, the most complex parts are located in terms $e_1$ and $e_2$; and, in fact, they make the formula much more complicated than the others. Similarly to the situation in Sec. 3, we can find an additional explicit formula for $\Delta(x^n+ax^{n-1}+bx^{n-3}+c)$.

\begin{wn}
The discriminant of a quadrinomial $f(x)=x^n+ax^{n-1}+bx^{n-3}+c$ such that $abc\neq 0$ and $n>5$ is given by
\begin{align*}
    \Delta(f)=&(-1)^{\frac{(n-1)(n+2)}{2}}b^{n-3}c^{n+1}(n-3)^{n-3}\left(\frac{3b}{c}+nt_{n-2}\right)^3\\
    &\Bigg(\left(\frac{b(n-3)}{cn}\right)^2e_2^3+\frac{ab(n-1)(n-3)}{c^2n^2}(e_1^2-2e_2)e_2\\
    &+\frac{b(n-3)}{c^2n}(e_1^3-3e_1e_2) +\left(\frac{a(n-1)}{cn}\right)^2e_2+\frac{a(n-1)}{c^2n}e_1+\frac{1}{c^2}\Bigg),
\end{align*}
where $e_1=-\frac{cnt_{n-1}}{3b+cnt_{n-2}}, e_2= \frac{ab(n-3)-cn^2t_{n-3}}{b(n-3)(3b+cnt_{n-2})}$
and the sequence $(t_r)$ is expressed as
\begin{align*}
t_r = \begin{cases}
          \left(\frac{cn}{b(n-3)}\right)^{\ceil{\frac{r}{2}}}\sum_{k=0}^{\floor{\frac{r}{6}}}\binom{\floor{\frac{r}{2}}-k}{\floor{\frac{r}{2}}-3k}\left(\frac{1}{c}\right)^{2k}\left(\frac{a(n-1)}{cn}\right)^{\floor{\frac{r}{2}}-3k}\left(\frac{b(n-3)}{cn}\right)^k, & 2\nmid r, \\
          \left(\frac{cn}{b(n-3)}\right)^{\frac{r}{2}}\sum_{k=1}^{\floor{\frac{r+2}{6}}}\binom{\frac{r}{2}-k}{\frac{r}{2}-3k+1}\left(\frac{1}{c}\right)^{2k-1}\left(\frac{a(n-1)}{cn}\right)^{\frac{r}{2}-3k+1}\left(\frac{b(n-3)}{cn}\right)^{k-1}, & 2|r.
       \end{cases}
\end{align*}
\end{wn}

\begin{proof}
The procedure goes in the analogous way to the proof of Corollary $3.3$.
\end{proof}

\begin{uw}
\rm{It is worth to notice that an effective formula for $\Delta(x^n+ax^3+bx^2+c)$ might probably be derived in a very similar way to $\Delta(x^n+ax^3+bx+c)$. Clearly, Lemma $5.2$ needs to be modified in this case, and Theorem $2.6$ may be applied, but the remaining reasoning should remain actually the same.}
\end{uw}

\section{Concluding remarks}

A question arises whether the presented approach can be used to get a formula for the discriminant of a general quadrinomial. Let us consider a heuristic reasoning for a polynomial $p(x)=x^n+ax^k+bx^l+c$ such that $n>k>l>0$ and $abc \neq0$. We have
$$p(x)=q_0(x)p'(x)+r_0(x),$$
where
\begin{align*}
r_0(x) = \begin{cases}
          \frac{a(n-k)}{n}x^k+\frac{b(n-l)}{n}x^l+c, & k<n-1, \\
          -\frac{a^2(n-1)}{n^2}x^{n-2}+\frac{b(n-l)}{n}x^l-\frac{abl}{n^2}x^{l-1}+c, & k=n-1,
       \end{cases}
\end{align*}
and
\begin{align*}
q_0(x) = \begin{cases}
          \frac{1}{n}x, & k<n-1, \\
          \frac{1}{n}x+\frac{a}{n^2}, & k=n-1.
       \end{cases}
\end{align*}
Since
$$p'(x)=nx^{l-1}\left(x^{n-l}+\frac{ak}{n}x^{k-l}+\frac{bl}{n}\right),$$
it is convenient to use the multiplicative resultant property (Theorem $2.6$). Additionally, let us denote by $p_1'$ and $p_2'$ the polynomials $p_1'(x)=nx^{l-1}$ and $p_2'(x)=p'(x)/p_1'(x)$. The most demanding challenge in further examination relies on determining the resultant $R(p_2',r_0)$. We continue the polynomial division process as long as the degree of the remainder is greater than two. The length of this process depends on the value of $\min \{n-l,k\}$. In the first step, we need to find the explicit form of the sequence $(t_r)$, what is a highly non-trivial task in this case. Moreover, if we determined the formula of remainder $r_1$ from dividing either $p_2'$ by $r_0$ or $r_0$ by $p_2'$ (what depends on the $\min \{n-l,k\}$ again), we would proceed computing remainders from dividing $r_{j-1}$ by $r_j$ for $j=1,\ldots$ systematically until the remainder is a quadratic polynomial. Therefore, the complexity of the general problem increases and may be difficult to obtain a general result. However, for some infinite families of polynomials our method is effective. Indeed, let us consider \linebreak a quadrinomial $f(x)=x^{2n}+ax^n+bx^l+c$ such that $abc\neq 0$, $a^2\neq4c$ and $n>2l$. Since $f'(x)=2nx^{2n-1}+anx^{n-1}+blx^{l-1}$, we factor $f'$ into $f_1'(x):=2nx^{l-1}$ and $f_2':=f'/f_1'$. Now, we use Theorem $2.8$ and Theorem $2.6$ in order to obtain
\begin{align*}
    \Delta(f)=(-1)^{n(2n-1)}R(f,f_1')R(f,f_2').
\end{align*}
Theorem $2.5$ implies that $R(f,f_1')=(2n)^{2n}c^{l-1}$. On the other hand, the calculation of $R(f,f')$ is more difficult. To achieve the goal we use Theorem $2.4$, Corollary $2.2$ and Theorem $2.6$, respectively, and repeat such a procedure three times. It is worth mentioning that the multiplicative property of the resultant is responsible only for preserving monic polynomials as arguments of $R$. After the processes of dividing the polynomial in the first position of $R$ by the one in the second --- which is possible due to successive applying Corollary $2.2$ --- we get the following remainders:
\begin{align*}
    \widetilde{r}_0(x)&=\frac{a}{2}x^n+\frac{b(2n-l)}{2n}x^l+c,\\
    \widetilde{r}_1(x)&=\frac{a^2-4c}{2}x^{n-l}+\left(\frac{b(2n-l)}{an}\right)^2x^l+\frac{b\{a^2l+4c(2n-l)\}}{2a^2n},\\
    \widetilde{r}_2(x)&=-\frac{2}{a(a^2-4c)}\left(\frac{b(2n-l)}{n}\right)^2x^{2l}+\frac{b\{(2n-l)(a^2-8c)-a^2l\}}{a(a^2-4c)n}x^l+\frac{2c}{a}.
\end{align*}
We denote by $\overline{r}_i$ the polynomial $\widetilde{r}_i$ divided by its leading coefficient. Hence, we obtain that
\begin{align*}
    \Delta(f)&=(-1)^{n}(2n)^{2n}c^{l-1}\left(\frac{a}{2}\right)^{2n-l}(-1)^{n^2+n(n-l)}\left(\frac{a^2-4c}{2a}\right)^{n}\\
    &\phantom{(}\cdot(-1)^{(n-2l)(n-l)}\left(-\frac{2}{a(a^2-4c)}\left(\frac{b(2n-l)}{n}\right)^2\right)^{n-l}R(\overline{r}_1,\overline{r}_2).
\end{align*}
However, let us observe that $\overline{r}_2$ is a quadratic trinomial in $x^l$. Therefore, we can determine its roots easily, and apply Theorem $2.5$ in order to deduce that
\begin{align*}
    \Delta(f)=(-1)^{n+l}b^{2(n-l)}c^{l-1}(a^2-4c)^ln^{2l}(2n-l)^{2(n-l)}\prod_{j=1}^{2l}\overline{r}_1(\xi_j),
\end{align*}
where each $\xi_j$ is a root of $\overline{r}_2$ for any $j\in\{1,2,\ldots,2l\}$. Finally, if we fixed the value of $l$ and made similar computations to those from the preceding sections, we would find an effective formula for $\Delta(x^{2n}+ax^n+bx^l+c)$. 

Summing up, despite the fact that our method is probably too complex to find an explicit formula for $\Delta(x^n+ax^k+bx^l+c)$ in general, it can be successfully applied for some infinite families of quadrinomials. 

\section*{Acknowledgements}
I would like to thank Maciej Ulas and Piotr Miska for their profound comments and helpful suggestions. 


\end{document}